\documentclass[reqno]{amsart}

\usepackage{amsfonts}
\usepackage{amsmath}
\usepackage{amssymb}

\usepackage{cite}

\allowdisplaybreaks

\newtheorem{theorem}{Theorem}
\newtheorem{corollary}[theorem]{Corollary}
\newtheorem{definition}[theorem]{Definition}
\newtheorem{lemma}[theorem]{Lemma}
\newtheorem{proposition}[theorem]{Proposition}
\newtheorem{remark}[theorem]{Remark}

\AtBeginDocument{{\noindent\small
This is a preprint of a paper whose final and definite form is with 
\emph{Computers and Mathematics with Applications}, ISSN 0898-1221.
Submitted 8-Sept-2018; revised 16-March-2019; accepted for publication 19-March-2019.}
\vspace{9mm}}

\begin{document}

\title[Optimal control of a nonlocal thermistor problem]{%
Optimal control of a nonlocal thermistor problem with ABC fractional time derivatives}

\author[M. R. Sidi Ammi]{Moulay Rchid Sidi Ammi}
\address{M. R. Sidi Ammi: Department of Mathematics,
AMNEA Group, Faculty of Sciences and Techniques,
Moulay Ismail University,
B.P. 509, Errachidia, Morocco.}
\email{sidiammi@ua.pt, rachidsidiammi@yahoo.fr}
\urladdr{http://orcid.org/0000-0002-4488-9070}

\author[D. F. M. Torres]{Delfim F. M. Torres}
\address{D. F. M. Torres:
Center for Research and Development in Mathematics and Applications (CIDMA),
Department of Mathematics, University of Aveiro, 3810-193 Aveiro, Portugal.}
\email{delfim@ua.pt}
\urladdr{http://orcid.org/0000-0001-8641-2505}

% ---------------------------------------------

\date{Submitted 8-Sept-2018; Revised 16-March-2019; Accepted 19-March-2019}



% ---------------------------------------------

\begin{abstract}
We study an optimal control problem associated 
to a fractional nonlocal thermistor problem involving
the ABC (Atangana--Baleanu--Caputo) fractional time derivative. 
We first prove the existence and uniqueness of solution.
Then, we show that an optimal control exists.
Moreover, we obtain the optimality system 
that characterizes the control.
\end{abstract}

\maketitle

% ---------------------------------------------

\section{Introduction}
\label{sec:1}

Fractional calculus is a powerful mathematical tool  
to describe real-world phenomena with memory effects, 
being used in many scientific fields.
Many published works in fractional calculus put emphasis 
on the Riemann--Liouville power-law differential operator;
others suggest different fractional approaches of mathematical
modeling to represent physical problems, calling attention 
that a singularity on the power law leads to models that are
singular, which is not convenient for those with no sign of singularity. 
In particular, several applications of the exponential kernel suggested by 
Caputo and Fabrizio can be found in chemical reactions, electrostatics, 
fluid dynamics, geophysics and heat transfer \cite{MR3832160,MR3828262}.

If an object at one temperature is exposed to a medium with another temperature, 
the temperature difference between the object and the medium follows an exponential decay, 
according with Newton's law of cooling. Other examples may be found in luminescence, 
pharmacology and toxicology, physical optics, radioactivity and thermo-electricity, 
where there is a decline in resistance of a negative temperature coefficient thermistor, 
as the temperature, vibrations, finance or some other aspect is increased.
The generalized Mittag--Leffler function, considered as a generalization 
of the exponential decay and as power-law asymptotic for a very large time, 
occurs to handle non-locality and avoid singularity \cite{gorenflo}. 
According to Rudolf Gorenflo (1930--2017) \cite{gorenflo}, 
one can say that the Mittag--Leffler function is a practical memory function 
in several physical problems. It can be used as a waiting-time distribution, 
as well as a first-passage-time distribution for renewal processes \cite{gorenflo}.
Recently, such considerations lead to the introduction of ABC 
(Atangana--Baleanu--Caputo) fractional operators \cite{MR3807252,MR3856675}.

The Riemann--Liouville fractional derivative seems not the most appropriate 
to describe diffusion at different scales. Thanks to the non-obedience  
of commutativity and associativity criteria, and due to  Mittag--Leffler memory, 
the ABC fractional derivative promises to be a powerful mathematical tool, 
allowing to describe heterogeneity and diffusion at different scales, 
distinguishing between dynamical systems taking place at different scales 
without steady state. Here, we are interested to study an optimal control 
problem to the following nonlocal parabolic boundary value problem:
\begin{equation}
\label{eq1}
\begin{split}
^{ab}_{0}D_{t}^{\alpha}u- \triangle u 
&=  \frac{\lambda f(u)}{\left(
\int_{\Omega} f(u)\, dx\right)^{2}}\   
\mbox{ in } Q_{T}=\Omega \times (0, T) \,, \\
\frac{\partial u}{\partial\nu} 
&=  -\beta u \    \mbox{ on } 
S_{T}=\partial \Omega \times (0, T) \,, \\
u(0)&= u_{0} \     \mbox{ in } \Omega,
\end{split}
\end{equation}
where $^{ab}_{0}D_{t}^{\alpha}$, $\alpha \in (0, 1)$, is the Atangana--Baleanu 
fractional derivative of order $\alpha$ in the sense of Caputo  
with respect to time $t$ \cite{atangana},  $\triangle$ is the Laplacian 
with respect to the spacial variables, defined on 
$H^{2}(\Omega)\bigcap H_{0}^{1}(\Omega)$,
$f$ is a smooth function prescribed
below, and $T$ is a fixed positive real. The domain $\Omega$ is bounded 
in $\mathbb{R}^{N}$, $N \geq 1$, with a sufficiently smooth boundary 
$\partial \Omega$ and $Q_{T}=\Omega \times (0, T)$. Here, $\nu$ denotes the
outward unit normal and $\frac{\partial }{\partial \nu}= \nu \cdot \nabla$ 
is the normal derivative on $\partial \Omega$. Such
problems arise in many applications, for instance, in studying the
heat transfer in a resistor device whose electrical conductivity
$f$ is strongly dependent on the temperature $u$ (thermistors). 
When $\alpha = 1$, equation \eqref{eq1} describes the diffusion 
of the temperature $u$ generated by the electric current with the
presence of a nonlocal term \cite{lac2,tza,mac,psx,kw,MR2805614,sidel}. 
Constant $\lambda$ is a dimensionless parameter while function
$\beta$ is the positive thermal transfer coefficient. The
given value $u_{0}$ is the initial condition for the temperature. 
Mixed boundary conditions of Robin's type are considered, 
which are derived from Newton's cooling law.

Optimal control of problems governed by partial differential
equations occurs more and more frequently in different research
areas \cite{Lions71,MR2747291,MR2583281,MR2405382}.
Researchers are interested, essentially, to existence, 
regularity, and uniqueness of the optimal control problem, 
as well as necessary optimality conditions.
The optimal control theory for systems of thermistor problems 
with integer-order derivatives on time $\partial_{t}$ has been 
developed in \cite{gc,LeeShilkin,vsv,vol}.
Works on control theory applied to fractional differential equations, 
where the fractional time derivative is considered in Riemann--Liouville 
and Caputo senses, have been already studied \cite{MyID:410}. However, 
to the best of our knowledge, the use of the Atangana--Baleanu derivative 
is underdeveloped in this area. Particularly, we are not aware 
of any paper investigating the optimal control of \eqref{eq1}.
In our work, we choose the heat transfer coefficient $\beta$ 
as a control, because it plays a crucial role in the temperature 
variations of a thermistor \cite{ffh,zw,MR2454231}.

Our manuscript is organized as follows. In Section~\ref{sec:2}, 
we briefly collect definitions and preliminary results 
about fractional derivatives. Section~\ref{sec3} is devoted 
to the existence and uniqueness results for \eqref{eq1}, 
while in Section~\ref{sec4} we investigate the
corresponding control problem. Main results characterize, 
explicitly, the optimal control, extending those of \cite{sidel1,vol}.

% ----------------------------------------------

\section{Preliminary results}
\label{sec:2}

Our main goal consists to find a control $\beta$ belonging
to the set 
$$
U_{M}= \left\{ \beta \in L^{\infty}(\Omega \times (0, T))\, , 
0 < m \leq \beta \leq M \right\}
$$
of admissible controls, which minimizes the cost functional 
\begin{equation*}
J(\beta)= \int_{Q_{T}} u dx dt + \int_{S_{T}} \beta^{2} ds dt
\end{equation*}
defined in terms of $u(\beta)$ and $\beta$.
Precisely, we purpose to find $\overline{\beta} \in U_{M}$ such that
\begin{equation}
\label{P}
J(\overline{\beta})= \min_{\beta \in U_{M}} J(\beta).
\end{equation}

We now recall some properties on the Mittag--Leffler function 
and the definition of ABC fractional time derivative.
First, we define the two-parameter Mittag--Leffler function 
$E_{\alpha, \beta}(z)$, as the family of entire functions 
of $z$ given by
$$
E_{\alpha, \beta}(z) = \sum_{k=0}^{\infty} 
\frac{z^{k}}{\Gamma(k\alpha + \beta)}, 
\quad z \in \mathbb{C},
$$
where $\Gamma(\cdot)$ denotes the Gamma function
$$
\Gamma(z)= \int_{0}^{\infty} t^{z-1} e^{-t} dt, 
\quad Re(z) > 0.
$$
Observe that the exponential function is a particular 
case of the Mittag--Leffler function: $E_{1, 1}(z)=e^{z}$.
Follows the definition of fractional derivative in the sense 
of Atangana--Baleanu \cite{atangana1,atangana2}.

\begin{definition}
\label{def:1.1}
For a given function $u \in H^{1}(a, T)$, $T>a$, 
the Atangana--Baleanu fractional derivative in Caputo sense, 
shortly called the ABC fractional derivative, 
of $u$ of order $\alpha$ with base point $a$, 
is defined at a point $t \in (a, T)$ by
\begin{equation} 
\label{eq3}
^{ab}_{a}D_{t}^{\alpha} g(t)=\frac{B(\alpha)}{1-\alpha} 
\int_{a}^{t} u'(\tau) E_{\alpha, \alpha}[-\gamma (t-\tau)^{\alpha}] d\tau,
\end{equation}
where $\gamma= \frac{\alpha}{1-\alpha}$, $E_{\alpha, \alpha}$ 
stands for the Mittag--Leffler function, and 
$B(\alpha)= (1-\alpha) + \frac{\alpha }{\Gamma(\alpha)}$.
Furthermore, the Atangana--Baleanu fractional integral
of order $\alpha$ with base point $a$ is defined as
\begin{equation}
\label{eq4bis}
I_{t}^{\alpha} g(t)= \frac{1-\alpha}{B(\alpha)} g(t)+\frac{\alpha}{B(\alpha) 
\Gamma(\alpha)}\int_{a}^{t} g(t) (t-\tau)^{\alpha-1} d\tau.
\end{equation}
\end{definition}

\begin{remark}
For $\alpha =1$ in \eqref{eq3}, we obtain the usual ordinary  
derivative $\partial_{t}$. If $\alpha =0, 1$ in \eqref{eq4bis}, 
then we get the initial function and the classical integral, respectively.
\end{remark}

Rougly speaking, the following result asserts that going backwards 
in time with the fractional time derivative with nonsingular Mittag--Leffler 
kernel at the based point $T$ is equivalent as going forward in time 
with the fractional time derivative operator 
with nonsingular Mittag--Leffler kernel.

\begin{lemma}
\label{lem:}
Let $\eta: [0, T] \rightarrow \mathbb{R}$. Then, 
for all $\alpha \in (0, 1)$, the equivalence relation 
\begin{equation*}
^{ab}_{T}D_{t}^{\alpha} \eta(T-t)= ^{ab}_{0}D_{t}^{\alpha}\eta (t)
\end{equation*}
holds.
\end{lemma}

\begin{proof}
Follows directly from definition by change of variables.
\end{proof}

Along the paper, we always assume that the integrals exist.
Moreover, we consider the following assumptions:
\begin{itemize}
\item[(H1)] $f: \mathbb{R} \rightarrow \mathbb{R}$ is a positive
Lipshitzian continuous function;

\item[(H2)] there exist positive constants $c_{1}$ and $c_{2}$ such
that $c_{1} \leq f(\xi) \leq c_{2}$ $\forall$ $\xi \in \mathbb{R}$;

\item[(H3)] $u_{0} \in L^{2}(\Omega)$.
\end{itemize}

\begin{definition}
We say that $u$ is a weak solution to \eqref{eq1} if
\begin{equation}
\label{equa1}
\int_{\Omega} \, (^{ab}_{0}D_{t}^{\alpha} u) v dx 
+ \int_{\Omega} \nabla u \nabla v dx 
+ \int_{\partial \Omega} \beta u v ds
= \frac{\lambda }{\left(\int_{\Omega} f(u)\, dx\right)^{2}} 
\int_{\Omega} f(u) v dx 
\end{equation}
for all $v \in H^{1}(\Omega)$.
\end{definition}

\begin{proposition}
\label{prop4}
Let $u, v \in \mathbb{C^{\infty}}(\overline{Q_{T}})$. Then, 
\begin{equation*}
\begin{split}
\int_{\Omega}&\int_{0}^{T} \left (  ^{ab}_{0}D_{t}^{\alpha}u 
- \triangle u \right ) v dx dt\\
&=  \int_{0}^{T}\int_{\partial \Omega}  
u \frac{\partial v}{\partial \nu}  ds dt
- \int_{0}^{T}\int_{\partial \Omega}  v  
\frac{\partial u}{\partial \nu} ds dt \\
&\quad - \frac{B(\alpha)}{1-\alpha} \int_{\Omega} 
\int_{0}^{T} u(x, 0) E_{\alpha, \alpha}[-\gamma t^{\alpha}] v dx dt\\
&\quad + \int_{0}^{T}\int_{\Omega} u \left( 
- ^{ab}_{T}D_{t}^{\alpha} v - \triangle v \right) dx dt\\
&\quad +\frac{B(\alpha)}{1-\alpha} \int_{\Omega} v(x, T) 
\int_{0}^{T}  u E_{\alpha, \alpha}[-\gamma (T-t)^{\alpha}] dt dx.
\end{split}
\end{equation*}
\end{proposition}

\begin{proof}
From integration by parts involving the ABC fractional-time derivative 
(see \cite{djida}), a straightforward calculation gives that
\begin{multline}
\label{eq8}
\int_{0}^{T} \, ^{ab}_{0}D_{t}^{\alpha}u \cdot v dt 
=- \int_{0}^{T} \, ^{ab}_{T}D_{t}^{\alpha}v \cdot u dt
+ \frac{B(\alpha)}{1-\alpha} v(x, T)  
\int_{0}^{T} u E_{\alpha, \alpha}[-\gamma (T-t)^{\alpha}] dt\\
- \frac{B(\alpha)}{1-\alpha} u(x, 0) 
\int_{0}^{T} E_{\alpha}[-\gamma t^{\alpha}] v dt
\end{multline}
and
\begin{multline}
\label{eq9}
-\int_{\Omega}\int_{0}^{T} \triangle u \cdot v dx dt 
= \int_{0}^{T}\int_{\partial \Omega} u \frac{\partial v}{\partial \nu}  ds dt
- \int_{0}^{T}\int_{\partial \Omega} v \frac{\partial u}{\partial \nu} ds dt\\
- \int_{\Omega}\int_{0}^{T} \triangle v \cdot u dx dt.
\end{multline}
Combining \eqref{eq8} and \eqref{eq9}, we get the desired result.
\end{proof}

Using the boundary conditions of problem \eqref{eq1}, 
we immediately get the following corollary.

\begin{corollary}
\label{cor5}
Let $u, v \in \mathbb{C^{\infty}}(\overline{Q_{T}})$. Then,
\begin{equation*}
\begin{split}
\int_{\Omega}&\int_{0}^{T} 
\left (  ^{ab}_{0}D_{t}^{\alpha}u  - \triangle u \right ) v dx dt
= \int_{0}^{T}\int_{\partial \Omega}  \beta u v ds dt
+ \int_{0}^{T}\int_{\partial \Omega} u \frac{\partial v}{\partial \nu}  ds dt \\
&\quad - \frac{B(\alpha)}{1-\alpha} \int_{\Omega} \int_{0}^{T} u(x, 0) 
E_{\alpha, \alpha}[-\gamma t^{\alpha}] + \int_{0}^{T}\int_{\Omega} 
u \left ( - ^{ab}_{T}D_{t}^{\alpha} v - \triangle v   \right ) dx dt\\
&\quad +\frac{B(\alpha)}{1-\alpha} \int_{\Omega} v(x, T) 
\int_{0}^{T}  \psi E_{\alpha, \alpha}[-\gamma (T-t)^{\alpha}] dt dx.
\end{split}
\end{equation*}
\end{corollary}

Along the text, constants $c$ are generic, and may change at each occurrence.

% --------------------------------------------------

\section{Existence and uniqueness for \eqref{equa1}} 
\label{sec3}

We proceed similarly as in \cite{djida}. Let $V_{m}$ define a subspace 
of $H^{1}(\Omega)$ generated by $w_{1}$, $w_{2}$, $\ldots$, $w_{m}$, 
space vectors of orthogonal eigenfunctions of the operator $\Delta$. 
We seek $u_{m}: t \in (0,  T] \rightarrow u_{m}(t) \in V_{m}$, 
solution of the fractional differential equation
\begin{equation*}
\begin{cases}
\displaystyle \int_{\Omega} \, ^{ab}_{0}D_{t}^{\alpha} u_{m} v dx 
+ \int_{\Omega} \nabla u_{m} \nabla v dx 
+ \int_{\partial \Omega} \beta u_{m} v ds  
= (g(u_{m}), v) & \mbox{ for all } v \in V_{m}, \\
u_{m}(x, 0)=u_{0m} & \mbox{ for } x \in \Omega,
\end{cases}
\end{equation*}
with $g(u)= \frac{\lambda f(u)}{\left(\int_{\Omega} f(u)\, dx\right)^{2}}$.

\begin{theorem}
\label{thm6}
Let $\alpha \in (0, 1)$. Assume that $f \in L^{2}(Q_{T})$, 
$u_{0} \in L^{2}(\Omega)$. Let $(\cdot, \cdot)$ be the 
scalar product in $L^{2}(\Omega)$ and $a(\cdot, \cdot)$ 
be the bilinear form in $H^{1}_{0}(\Omega)$ defined by
$$
a(\phi, \psi)= \int_{\Omega} \nabla \phi(x) \nabla \psi(x) dx 
\quad \forall \phi, \psi \in H^{1}(\Omega).
$$
Then the problem
\begin{equation*}
\begin{cases}
\left(^{ab}_{0}D_{t}^{\alpha}u, v\right) + a(u(t), v)
= (f(t), v), & \mbox{ for all } t \in (0, T), \\
u(x, 0)=u_{0}, & \mbox{ for } x \in \Omega,
\end{cases}
\end{equation*}
has a unique solution $u \in L^{2}(0, T, H_{0}^{1}(\Omega)) 
\bigcap \mathbb{C}(0, T, H_{0}^{1}(\Omega))$ given by
\begin{multline}
\label{eq13}
u(x,t)=  
\sum_{i=1}^{+\infty}  \left( \zeta_{i} E_{\alpha}[-\gamma_{i} 
t^{\alpha}]u^{0}_{i}+\frac{(1-\alpha)\zeta_{i}}{B(\alpha)} f_{i}(t) \right.\\
\left.+  K_{i}\int_{0}^{t}(t-s)^{\alpha-1} E_{\alpha, \alpha}\left[
-\gamma_{i}(t-s)^{\alpha}\right] f_{i}(s) ds  \right) w_{i},
\end{multline}
where $\gamma_{i}$ and $\zeta_{i}$ are constants. 
Moreover, provided $u_{0} \in L^{2}(\Omega)$,
$u$ satisfies the inequalities
\begin{equation}
\label{eq14}
\|u\|_{ L^{2}(0, T, H_{0}^{1}(\Omega))} 
\leq \mu_{1} ( \|u_{0}\|_{H_{0}^{1}(\Omega)} +\|f\|_{L^{2}(Q_{T})})
\end{equation}
and
\begin{equation}
\label{eq15}
\|u\|_{ L^{2}(\Omega))} \leq \mu_{2} ( \|u_{0}\|_{L^{2}(\Omega)} +\|f\|_{L^{2}(Q_{T})}),
\end{equation}
where $\mu_{1}$ and $\mu_{2}$ are positive constants.
\end{theorem}

\begin{proof}
Because $u_{m}(t) \in V_{m}$, one has
$$
u_{m}(t)=\sum_{i=1}^{m} (u(t), w_{i}) w_{i}
= \sum_{i=1}^{m} u_{i}(t) w_{i}.
$$
The fact that $g(u) \in L^{2}(Q_{T})$ implies that $u_{m}$ can be 
written in explicit form (see \eqref{eq13}). Arguing exactly as in \cite{djida}, 
we can prove that $u_{m}(t)$ is a Cauchy sequence in the space $L^{2}(0, T, H^{1}(\Omega))$ 
and $\mathbb{C}(0, T, L^{2}(\Omega)$. Using the estimates 
\eqref{eq14}--\eqref{eq15} of Theorem~\ref{thm6}, we have that
\begin{equation*}
\begin{gathered}
u_{m} \rightarrow u \mbox{ weakly in } L^{2}(0, T,
H^{1}(\Omega)),\\
\frac{\partial u_{m}}{\partial t} \rightarrow \frac{\partial
u_{m}}{\partial t}  \mbox{ weakly in } L^{2}(0, T,
H^{-1}(\Omega)),\\
u_{m} \rightarrow u \mbox{ weakly in } \mathbb{C}(0, T,
L^{2}(\Omega)),\\
u_{m} \rightarrow u \mbox{ strongly in } L^{2}(Q_{T}),\\
u_{m} \rightarrow u \mbox{ a.e.  in }L^{2}(Q_{T}).
\end{gathered}
\end{equation*}
By standard techniques of Lebesgue's theorem 
and some compactness arguments of Lions \cite{jll}, 
one gets that $u(t)$ is a solution of problem \eqref{eq1}. 
Then, the existence and uniqueness result follows.
\end{proof}

% ----------------------------------------------

\section{Existence of an optimal control}
\label{sec4}

We prove existence of an optimal control by using
minimizing sequences.

\begin{theorem}
\label{thm:3.1} 
Assume that assumptions (H1)--(H3) are satisfied.
Then, there exists at least an optimal solution 
$\beta \in L^{\infty}(Q_{T})$
such that \eqref{P} holds true.
\end{theorem}

\begin{proof}
Let $(\beta_{n})_{n}$ be a minimizing sequence of
$J(\beta)$ in $ U_{M}$ such that
$$
\lim_{n \rightarrow +\infty} J(\beta_{n})
= \inf_{\beta \in U_{M}} J(\beta).
$$
Then, $u_{n}=u_{n}(x, t, \beta_{n})$, 
the corresponding solutions to \eqref{eq1}, satisfy
\begin{equation*}
\begin{split}
^{ab}_{0}D_{t}^{\alpha}u_{n}- \triangle u_{n} 
&= \frac{\lambda f(u_{n})}{(
\int_{\Omega} f(u_{n})\, dx)^{2}}\, ,  
\mbox{ in } Q_{T}= \Omega \times (0, T), \\
\frac{\partial u_{n}}{\partial\nu} 
&= -\beta_{n} u_{n} \, ,  \mbox{ on } S_{T}
=\partial \Omega \times (0, T), \\
u_{n}(0, x)&= u_{0}(x) \, ,  \mbox{ in } \Omega.
\end{split}
\end{equation*}
By Theorem~\ref{thm6}, we have that $(u_{n})$ is bounded, 
independently of $n$ in $L^{2}(0, T, H^{1}(\Omega))$.
Moreover, for a positive constant independent of $n$, we have
\begin{equation*}
\| ^{ab}_{0}D_{t}^{\alpha}u_{n}- \triangle u_{n} \|_{L^{2}(Q_{T})} \leq c.
\end{equation*}
Therefore, there exists $u$, for extracted sequences of $(u_{n})_{n}$, 
still denoted by $(u_{n})$, and there exists
$\beta \in U_{M}$ such that
\begin{equation}
\label{eq19}
\begin{gathered}
^{ab}_{0}D_{t}^{\alpha}u_{n}- \triangle u_{n} 
\rightharpoonup \delta \mbox{ weakly in } L^{2}(Q_{T}),\\
u_{n} \rightarrow u \mbox{ weakly in } L^{2}(0, T,
H^{1}(\Omega)),\\
\frac{\partial u_{n}}{\partial t} \rightarrow \frac{\partial
u}{\partial t}  \mbox{ weakly in } \mathbb{D}'(Q_{T}) 
\mbox{ and } L^{2}(0, T, H^{-1}(\Omega)),\\
u_{n} \rightarrow u \mbox{ weakly in } L^{\infty}(0, T,
L^{2}(\Omega)) \mbox{ and in }  L^{2}(Q_{T}),\\
u_{n} \rightarrow u \mbox{ strongly in } L^{2}(Q_{T}),\\
u_{n} \rightarrow u \mbox{ a.e.  in }L^{2}(Q_{T}),\\
\beta_{n} \rightarrow \beta \mbox{ weakly in } L^{2}(\partial \Omega),\\
\beta_{n} \rightarrow \beta \mbox{ weakly star in } L^{\infty}(\partial \Omega),
\end{gathered}
\end{equation}
where $\mathbb{D}'(Q_{T})$ is the dual of $\mathbb{D}(Q_{T})$, 
the set of $\mathbb{C}^{\infty}$ functions on $Q_{T}$ with compact support.
One can prove that
$$
^{ab}_{0}D_{t}^{\alpha}u_{n}- \triangle u_{n} 
\rightarrow ^{ab}_{0}D_{t}^{\alpha}u 
- \triangle u \mbox{ weakly in } \mathbb{D}'(Q_{T}).
$$
Indeed, we have
$$
\int_{0}^{T} \int_{\Omega} u_{n}( ^{ab}_{0}D_{t}^{\alpha}v 
- \triangle v ) dx dt
\rightarrow
\int_{0}^{T} \int_{\Omega} u( -^{ab}_{T}D_{t}^{\alpha}v -
 \triangle v ) dx dt, \forall v \in \mathbb{D}(Q_{T})
$$
and
$$
\int_{\Omega} v(x, T) \int_{0}^{T} u_{n} 
E_{\alpha, \alpha}[-\gamma (T-t)^{\alpha}] dt dx
\rightarrow
\int_{\Omega} v(x, T) \int_{0}^{T} u E_{\alpha, \alpha}[
-\gamma (T-t)^{\alpha}] dt dx.
$$
We now prove that for all $v \in H^{1}(\Omega)$ and
$n \rightarrow \infty$ one has
$$
\int_{\partial \Omega} \beta_{n} u_{n} v ds  \rightarrow
\int_{\partial \Omega} \beta u v ds.
$$
In fact, 
\begin{equation}
\label{eq20}
\beta_{n} u_{n} v  - \beta u v  
= \beta_{n}( u_{n}  -  u )v +  (\beta_{n}- \beta) uv.
\end{equation}
By using that $\beta_{n}$ is essentially bounded, 
Schwartz's inequality and the trace inequality 
$\|u\|_{L^{2}(\partial \Omega)} \leq c \|u\|_{H^{1}(\Omega)}$, 
it leads from limits \eqref{eq19} that the right-hand
side of \eqref{eq20} goes to $0$ when $n \rightarrow \infty$.
Thus,
$$
^{ab}_{0}D_{t}^{\alpha}u_{n}- \triangle u_{n} 
\rightarrow ^{ab}_{0}D_{t}^{\alpha}u 
- \triangle u \mbox{ weakly in } \mathbb{D}'(Q_{T}).
$$
From the uniqueness of the limit, we have
$$
^{ab}_{0}D_{t}^{\alpha}u - \triangle u = \delta.
$$
Since $u \in L^{2}(Q_{T})$ and $ ^{ab}_{0}D_{t}^{\alpha}u - \triangle u 
\in L^{2}(Q_{T})$, we know that $u/\partial \Omega$
and $\frac{\partial u}{\partial \nu}/\partial \Omega$ exist 
and belong to $H^{-\frac{1}{2}}(\partial \Omega)$ and  
$H^{-\frac{3}{2}}(\partial \Omega)$, respectively. It follows that
$$
\int_{\partial \Omega} u_{n} \frac{\partial v}{\partial \nu} 
\rightarrow \int_{\partial \Omega} u \frac{\partial v}{\partial \nu}
\quad \forall v \in \mathbb{D}(Q_{T}).
$$
On the other hand, we have $u_{n} \rightarrow u$ a.e. in
$\Omega \times (0, T)$. Since $f$ is continuous,
$f(u_{n}) \rightarrow f(u) \, a.e. \mbox{ in } L^{2}(\Omega)$.
It yields that
\begin{equation*}
\int_{\Omega} f(u_{n}) dx \rightarrow \int_{\Omega} f(u)
dx
\end{equation*}
and
\begin{equation*}
\int_{\Omega} f(u_{n})v dx \rightarrow \int_{\Omega} f(u) v dx, \,
\, \forall v \in H^{1}(\Omega).
\end{equation*}
By passing to the limit in the equation fulfilled by $u_{n}$, 
and using Corollary~\ref{cor5}, we deduce that $u$ is a solution 
of \eqref{eq1}. Finally, function $\beta \rightarrow J(\beta)$ 
is lower semi-continuous. Therefore,
$$
J(\beta) \leq \liminf_{n\rightarrow \infty} J(\beta_{n}),
$$
which implies that $J(\beta) = \inf_{\beta \in U_{M}} J(\beta)$. 
The uniqueness of $\beta$ comes from the strict convexity of functional $J$.
\end{proof}

% ----------------------------------------------

\section{Optimality conditions}
\label{sec:4}

In this section, our aim is to obtain optimality conditions.
As we shall see, our necessary optimality conditions involve 
an adjoint system defined by means of the backward Atangana--Baleanu
fractional-time derivative. To prove them, we assume, 
in addition to hypotheses (H1)--(H3), that
\begin{itemize}
\item[(H4)] $f$ is of class $C^{1}$.
\end{itemize}

Due to its dependence on $u$, the objective functional is differentiated with
respect to the minimizing control. We calculate the G\^{a}teaux derivative 
of $J$ with respect to the control $\beta$ in the direction $l$ at $\beta$. 
We also need to differentiate $u$ with respect to the control $\beta$. 
The difference quotient $\left(u(\beta + \varepsilon l)-u(\beta)\right)/\varepsilon$
is expected to converge weakly in $H^{1}(\Omega)$ to a function $\psi$ satisfying a
linear PDE, which leads to the adjoint system.

\begin{theorem}
\label{thm31}
Assume hypotheses (H1)--(H4). Then $\beta \mapsto u(\beta)$ is
differentiable in the sense that as $\varepsilon \rightarrow 0$
one has
$$
\frac{u(\beta + \varepsilon l)-u(\beta)}{\varepsilon}\rightarrow
\psi \mbox{ weakly in } H^{1}(\Omega)
$$
for any $\beta, l \in U_{M}$ such that 
$(\beta + \varepsilon l) \in U_{M}$ for small $\varepsilon$. 
Moreover, $\psi$ fulfills the following system:
\begin{equation}
\label{sys}
\begin{gathered}
^{ab}_{0}D_{t}^{\alpha} \psi - \triangle \psi 
= \frac{-2\lambda f(u)}{( \int_{\Omega} f(u)\, dx )^{3}}
\int_{\Omega}f'(u) \psi dx+ \frac{\lambda f'(u) \psi }{\left(
\int_{\Omega} f(u)\, dx\right)^{2}}
\quad \mbox{ in } \Omega, \\
\frac{\partial \psi}{\partial \nu}+ \beta \psi + l u = 0 
\mbox{ on } \partial \Omega.
\end{gathered}
\end{equation}
\end{theorem}

\begin{proof}
Denote $u=u(\beta)$ and $u_{\varepsilon}=u(\beta_{\varepsilon})$,
where $\beta_{\varepsilon}= \beta + \varepsilon l$.
Subtracting equation \eqref{eq1} from the corresponding equation
of $u_{\varepsilon}$, we have
\begin{equation*}
^{ab}_{0}D_{t}^{\alpha}\left(\frac{u_{\varepsilon}-u}{\varepsilon}\right) 
- \triangle\left(\frac{u_{\varepsilon}-u}{\varepsilon}\right)
= g(u_{\varepsilon})-g(u)
\end{equation*}
with
$$
g(u_{\varepsilon})-g(u)= \frac{\lambda}{\epsilon} \frac{\left (f(u_{\varepsilon})
-f(u)\right )}{( \int_{\Omega} f(u_{\varepsilon})\, dx)^{2}}
+ \frac{\lambda}{\epsilon} f(u)\left( \frac{1}{\left(\int_{\Omega}
f(u_{\varepsilon})\, dx\right)^{2}}
- \frac{1}{( \int_{\Omega} f(u)\, dx)^{2}}\right).
$$
As in the first section, since $g(u_{\varepsilon})-g(u) 
\in L^{\infty}(\Omega) \subseteq L^{2}(\Omega)$,
by using the energy estimates of Theorem~\ref{thm6}, we get that
$$
\frac{u_{\varepsilon}-u}{\varepsilon} 
\in L^{\infty}(0, T, L^{2}(\Omega)) \bigcap L^{2}(0, T, H^{1}(\Omega)).
$$
It follows, up to a subsequence of $\varepsilon$ which tends to $0$, 
that there exists $\psi$ such that
\begin{equation}
\label{eq23}
\begin{gathered}
\frac{u_{\varepsilon}-u}{\varepsilon} \rightarrow \psi \mbox{
weakly in } L^{\infty}(0, T, L^{2}(\Omega)),\\
\frac{u_{\varepsilon}-u}{\varepsilon} \rightarrow \psi \mbox{
weakly in } L^{2}(0, T, H^{1}(\Omega)),\\
\frac{\partial}{\partial t}\left(\frac{u_{\varepsilon}-u}{\varepsilon}\right)
\rightarrow \frac{\partial \psi}{\partial t} \mbox{ weakly in } L^{2}(0, T, H^{-1}\Omega)),\\
\frac{u_{\varepsilon}-u}{\varepsilon} \rightarrow \psi \mbox{
weakly in } L^{\infty}(0, T, L^{2}(\partial \Omega)),\\
\beta_{\varepsilon} \rightarrow \beta \mbox{ weakly in }
L^{2}(\partial \Omega),\\
\beta_{\varepsilon} \rightarrow \beta \mbox{ weakly in }
L^{\infty}(\Omega).
\end{gathered}
\end{equation}
From equations satisfied by $u_{\varepsilon}$ and $u$, 
and in view of Proposition~\ref{prop4}, we have that 
\begin{multline*}
\int_{\Omega}\int_{0}^{T} \, ^{ab}_{0}D_{t}^{\alpha}\left(
\frac{u_{\varepsilon}-u}{\varepsilon}\right) v dx dt
+ \int_{\Omega}\int_{0}^{T}\nabla\left(\frac{u_{\varepsilon}-u}{\varepsilon}\right) 
\nabla v dx dt \\
+ \int_{0}^{T}\int_{\partial \Omega} \frac{u_{\varepsilon}-u}{\varepsilon} v ds dt
+ \int_{0}^{T}\int_{\partial \Omega} lu_{\varepsilon} v ds dt = g(u_{\varepsilon})-g(u).
\end{multline*}
Using \eqref{eq23} and passing to the limit as $\varepsilon \rightarrow 0$, we get that
\begin{equation*}
\int_{\Omega}\int_{0}^{T} \, ^{ab}_{0}D_{t}^{\alpha} \psi 
+ \int_{\Omega}\int_{0}^{T}\nabla \psi \nabla v dx dt
+ \int_{0}^{T}\int_{\partial \Omega} (\beta \psi +lu ) ds dt 
=  \lim_{\varepsilon \rightarrow 0 }(g(u_{\varepsilon})-g(u)).
\end{equation*}
By Green's formula, it follows that
\begin{multline*}
\int_{\Omega}\int_{0}^{T} \left(  ^{ab}_{0}D_{t}^{\alpha}\psi 
- \triangle \psi \right) v dx dt \\
+ \int_{0}^{T}\int_{\partial \Omega}\left(\frac{\partial 
\psi}{\partial \nu}+\beta \psi +lu\right)v ds dt  
=\lim_{\varepsilon \rightarrow 0 }(g(u_{\varepsilon})-g(u)).
\end{multline*}
We conclude that $\psi$ satisfies the system
\begin{equation*}
\begin{gathered}
^{ab}_{0}D_{t}^{\alpha}\psi - \triangle \psi 
= \lim_{\varepsilon \rightarrow 0 }(g(u_{\varepsilon})-g(u)),\\
\frac{\partial \psi}{\partial \nu}+ \beta \psi + l u = 0
\quad \mbox{ on } \partial \Omega.
\end{gathered}
\end{equation*}
Set $g(u_{\varepsilon})-g(u)=(I) + (II)$ with
$$
(I) := \frac{\lambda}{\left(\int_{\Omega} f(u_{\varepsilon})\, dx\right)^{2}}
\int_{\Omega}\frac{f(u_{\varepsilon})-f(u)}{\varepsilon} \cdot v dx
$$
and
$$
(II) := \frac{\lambda}{\varepsilon} \left(
\frac{1}{\left(\int_{\Omega} f(u_{\varepsilon})\, dx\right)^{2}}
- \frac{1}{\left(\int_{\Omega} f(u)\, dx\right)^{2}}\right)
\int_{\Omega} f(u) v \, dx.
$$
We can reformulate $(II)$ as follows:
\begin{equation*}
\begin{split}
(II) &= \frac{\lambda}{\varepsilon} \frac{( \int_{\Omega} f(u)\, dx)^{2}
- ( \int_{\Omega} f(u_{\varepsilon})\, dx )^{2}}{(
\int_{\Omega} f(u)\, dx )^{2} ( \int_{\Omega} f(u_{\varepsilon})\,
dx )^{2}} \int_{\Omega}f(u) v dx\\
& = \lambda \int_{\Omega}
\frac{(f(u)-f(u_{\varepsilon}))}{\varepsilon}dx \times
\frac{\int_{\Omega}(f(u)+f(u_{\varepsilon}))dx}{( \int_{\Omega}
f(u_{\varepsilon})\, dx )^{2} ( \int_{\Omega} f(u)\, dx )^{2}}
\int_{\Omega}f(u) v dx \, .
\end{split}
\end{equation*}
Using the weak convergence \eqref{eq23}, we can prove that
\begin{equation*}
(II) \rightarrow \frac{-2 \lambda \int_{\Omega} f(u)v dx }{(
\int_{\Omega} f(u)\, dx )^{3}} \int_{\Omega} f'(u) \psi dx \mbox{
as } \varepsilon \rightarrow 0.
\end{equation*}
Similarly, 
\begin{equation*}
(I) \rightarrow \frac{\lambda}{( \int_{\Omega} f(u)\, dx )^{2}}
\int_{\Omega} f'(u) \psi v dx
 \mbox{ as } \varepsilon \rightarrow 0.
\end{equation*}
We conclude that $\psi$ verifies
\begin{equation*}
\begin{gathered}
^{ab}_{0}D_{t}^{\alpha}\psi - \triangle \psi 
= \frac{-2\lambda \int_{\Omega} f'(u) \psi dx}{( \int_{\Omega}
f(u)\, dx )^{3}} f(u) + \frac{\lambda f'(u) \psi}{( \int_{\Omega}
f(u)\, dx )^{2}} \quad \mbox{ in } \Omega, \\
\frac{\partial \psi}{\partial \nu}+ \beta \psi + l u = 0
\quad \mbox{ on } \partial \Omega.
\end{gathered}
\end{equation*}
This ends the proof of Theorem~\ref{thm31}.
\end{proof}

% ----------------------------------------------

\subsection{Derivation of the adjoint system}

To get the optimality system, we need first to derive 
the adjoint operator associated with $\psi$.
Let $v$ be an enough smooth function defined in $Q_{T}$.
By the first equation of \eqref{sys}, we have
\begin{multline*}
\int_{\Omega}\int_{0}^{T} \left(  ^{ab}_{0}D_{t}^{\alpha}\psi 
- \triangle \psi \right) v dx dt \\
=\frac{-2\lambda \int_{\Omega} f'(u) \psi dx}{\left( 
\int_{\Omega}f(u)\, dx\right)^{3}} \int_{Q_{T}} f(u) v dx dt 
+ \frac{\int_{Q_{T}}\lambda f'(u) \psi v dx dt}{\left(
\int_{\Omega} f(u)\, dx\right)^{2}} \mbox{ in } \Omega.
\end{multline*}
Integrating by parts, one has
\begin{equation*}
\begin{split}
\int_{\Omega}\int_{0}^{T} \, ^{ab}_{0}D_{t}^{\alpha}\psi. v dt dx 
&= -  \int_{\Omega} \int_{0}^{T} \, ^{ab}_{T}D_{t}^{\alpha}v. \psi dt dx\\
&\quad + \frac{B(\alpha)}{1-\alpha} \int_{\Omega} v(x, T)  
\int_{0}^{T} \psi E_{\alpha, \alpha}[-\gamma (T-t)^{\alpha}] dt.
\end{split}
\end{equation*}
Then,
\begin{multline*}
\int_{\Omega}\int_{0}^{T} \left(  
-^{ab}_{T}D_{t}^{\alpha}v - \triangle v \right ) \psi dx dt 
+ \frac{B(\alpha)}{1-\alpha} \int_{\Omega} v(x, T)  \int_{0}^{T} 
\psi E_{\alpha, \alpha}[-\gamma (T-t)^{\alpha}] dt\\
=\frac{-2\lambda \int_{\Omega} f'(u) \psi dx}{\left(\int_{\Omega}
f(u)\, dx\right)^{3}} \int_{Q_{T}} f(u) \varphi dx dt 
+ \frac{\int_{Q_{T}}\lambda f'(u) \psi \varphi dx dt}{\left(
\int_{\Omega} f(u)\, dx\right)^{2}} \mbox{ in } \Omega.
\end{multline*}
Introducing the boundary and initial conditions
$$
\frac{\partial v}{\partial \nu} + \beta v = 0
\mbox{ on } \partial \Omega \times (0, T),
\quad v(x, T) = 0,
$$
then function $v$ satisfies the adjoint system given by
\begin{equation}
\label{eq10}
\begin{gathered}
- ^{ab}_{T}D_{t}^{\alpha}v - \triangle v
= \frac{-2\lambda \int_{\Omega} f(u) \varphi dx}{( \int_{\Omega}
f(u)\, dx )^{3}} f'(u)+ \frac{\lambda  f'(u) \varphi}{(
\int_{\Omega} f(u)\, dx )^{2}}+1  \mbox{ in } Q_{T},\\
\frac{\partial v}{\partial \nu} + \beta v = 0
\mbox{ on } \partial \Omega \times (0, T),\\
v(T) = 0,
\end{gathered}
\end{equation}
where the $1$ appears from differentiation
of the integrand of $J(\beta)$ with respect
to the state $u$.

\begin{remark}
\label{rmk}
Given an optimal control $\beta \in U_{M}$ and the corresponding state $u$, 
the existence of solution to the adjoint system can be established by imposing 
additional regularity conditions on the electrical conductivity and following 
the same procedure we have followed for the existence results of \eqref{eq1}.
\end{remark}

% ----------------------------------------------

\subsection{Derivation of the optimality system}

Gathering equation \eqref{eq1} and the adjoint system
\eqref{eq10}, we obtain the following optimality system:
\begin{equation}
\label{eq11}
\begin{gathered}
u_{t}- \triangle u = \frac{\lambda f(u)}{\left(\int_{\Omega}
f(u)\, dx\right)^{2}}, \\
- ^{ab}_{T}D_{t}^{\alpha}v  - \triangle   v
= \frac{-2\lambda \int_{\Omega} f(u) v dx}{\left(\int_{\Omega}
f(u)\, dx\right)^{3}} f'(u)+ \frac{\lambda  f'(u) v}{\left(
\int_{\Omega} f(u)\, dx\right)^{2}}+1  \mbox{ in } Q_{T},\\
\frac{\partial u}{\partial \nu} + \beta u = 0
\mbox{ on } \partial \Omega \times (0, T),\\
\frac{\partial v}{\partial \nu} + \beta v = 0
\mbox{ on } \partial \Omega \times (0, T),\\
v(T)  = 0, \quad u(0)=u_0.
\end{gathered}
\end{equation}

\begin{remark}
The existence of solution to the optimality system \eqref{eq11}
follows from the existence of solution to the state system \eqref{eq1}
and the adjoint system \eqref{eq10}, combined with the existence of optimal
control.
\end{remark}

% --------------------------------------------------

\section{Conclusion}

In this paper we investigated an optimal control problem 
for a nonlocal thermistor problem  with a fractional time derivative
with nonlocal nonsingular Mittag--Leffler kernel. We proved existence
and uniqueness of the control. The optimality system describing the
optimal control was discussed. 

% ----------------------------------------------

\section*{Acknowledgements}

The authors were supported by the \emph{Center for Research
and Development in Mathematics and Applications} (CIDMA)
of University of Aveiro, through Funda\c{c}\~ao
para a Ci\^encia e a Tecnologia (FCT),
within project UID/MAT/04106/2019.

% ----------------------------------------------

% ----------------------------------------------

\end{document}